\documentclass[12pt]{amsart}
\usepackage{amscd}
\usepackage{amsmath}
\usepackage{verbatim}

\usepackage[all]{xy}

\usepackage{pb-diagram,pb-xy}

\usepackage{amssymb}

\numberwithin{equation}{section}

\newtheorem{thm}[equation]{Theorem}
\newtheorem{prop}[equation]{Proposition}

\newtheorem{lemma}[equation]{Lemma}

\theoremstyle{definition}

\newcommand{\Br}{\mathop{\mathrm{Br}}}

\newcommand{\ind}{\mathop{\mathrm{ind}}}

\newcommand{\cores}{\mathop{\mathrm{cor}}\nolimits}

\newcommand{\cdim}{\mathop{\mathrm{cdim}}\nolimits}

\newcommand{\F}{\mathbb{F}}

\newcommand{\Spec}{\operatorname{Spec}}

\newcommand{\Coprod}{\operatornamewithlimits{\textstyle\coprod}}

\newcommand{\CM}{\operatorname{CM}}

\renewcommand{\phi}{\varphi}

\newcommand{\RatM}{\dashrightarrow}

\newcommand{\WR}{\mathcal{R}}

\usepackage[hypertex]{hyperref}

\title
[Weil transfer of Severi-Brauer varieties]
{Incompressibility of quadratic Weil transfer
\\
of generalized Severi-Brauer varieties}

\keywords
{Algebraic groups,
projective homogeneous varieties,
Chow groups.\\
{\em Mathematical Subject Classification (2010):}
14L17; 14C25}

\author
{Nikita A. Karpenko}

\address
{UPMC Univ Paris 06\\
Institut de Math\'ematiques de Jussieu\\
F-75252 Paris\\
FRANCE}

\address
{{\it Web page:}
{\tt www.math.jussieu.fr/\~{ }karpenko}}

\email {karpenko {\it at} math.jussieu.fr}

\date
{28 October 2009. Revised: 26 September 2010}

\begin{document}

\begin{abstract}
Let $X$ be the variety obtained by the Weil transfer with respect to a quadratic separable field extension of a
generalized Severi-Brauer variety.
We study (and, in some cases, determine) the canonical dimension, incompressibility, and motivic indecomposability
of $X$.
We determine the canonical $2$-dimension of $X$ (in the general case).
\end{abstract}

\maketitle


\section
{Introduction}

The expression {\em canonical dimension} appeared for the first time in
\cite{MR2183253}.
The $p$-local version (where $p$ is a prime), called {\em canonical $p$-dimension}, comes from
\cite{MR2258262}.
They both can be introduced as particular cases of a (formally) older notion of the {\em essential \hbox{($p$-)}dimension}
(although, in fact, the canonical dimension has been implicitly studied for a long time before).
Below, we reproduce the modern definitions of \cite{ASM-ed}.

A connected smooth complete variety $X$ over a field $F$
is called {\em incompressible}, if any rational map $X\RatM X$ is dominant.

In most cases when it is known that a particular variety $X$ is incompressible, it is proved by establishing
that $X$ is {\em $p$-incompressible} for some positive prime integer $p$.
This is a stronger property which says that for any integral variety $X'$, admitting a dominant morphism
to $X$ of degree coprime with $p$, any morphism $X'\to X$ is dominant.

In some interesting cases, $p$-incompressibility of $X$ is, in its turn,  a consequence of a stronger property --
indecomposability of the {\em $p$-motive} of $X$.
Here by the $p$-motive we mean the classical Grothendieck motive of the variety constructed
using the Chow group with coefficients in the finite  field of $p$ elements.

{\em Canonical dimension} $\cdim X$ is a numerical invariant which measures the level of the incompressibility.
It is defined as the least dimension of the image of a rational map $X\RatM X$.
{\em Canonical $p$-dimension}, the $p$-local version, measures the $p$-incompressibility and is the least
dimension of the image of a morphism $X'\to X$, where $X'$ runs over the integral varieties admitting a dominant
morphism to $X$ of a $p$-coprime degree.
We always have the inequalities $\cdim_p X\leq\cdim X\leq\dim X$; the equality $\cdim X=\dim X$ means incompressibility
and the equality $\cdim_p X=\dim X$ means $p$-incompressibility.

Let $D$ be a central division $F$-algebra of degree a power $p^n$ of a prime $p$.
According to \cite{upper}, for any $i=0,\dots,n$, the (generalized Severi-Brauer) variety $X(p^i;D)$
(of the right ideals in $D$ of the reduced dimension $p^i$)
is $p$-incompressible.
Moreover, the $p$-motive of the variety $X(1;D)$ (this is the usual Severi-Brauer variety of $D$) is indecomposable.

In the case of the prime $p=2$,
these results have important consequences for orthogonal involutions on central simple algebras,
cf. \cite{hypernew}.
For a similar study of unitary involutions, one would need similar results on Weil transfers (with respect to separable
quadratic field extensions) of generalized Severi-Brauer varities.
Since a central simple algebra $A$ over a separable quadratic field extension $L/F$ admits a unitary
($F$-linear) involution
if and only if the norm algebra $N_{L/F}A$ is trivial, \cite[Theorem 3.1(2)]{MR1632779},
only the case of a trivial norm algebra is of interest from this viewpoint.

The main result of this paper is the following theorem, where we write
$\WR_{L/F}X$ for the Weil transfer of an $L$-variety $X$.

\begin{thm}
\label{thm1}
Let $F$ be a field, $L/F$ a quadratic separable field extension,
$n$ a non-negative integer,
and $D$ a central division $L$-algebra of degree $2^n$ such that the norm algebra $N_{L/F}D$
is trivial.
For any integer $i\in[0,\;n]$, the variety $\WR_{L/F} X(2^i;D)$ is $2$-incompressible.
\end{thm}

Moreover, in the case of a usual Severi-Brauer variety we also have the $2$-motivic indecomposability:

\begin{thm}
\label{thm2}
In the settings of Theorem \ref{thm1},
the $2$-motive of the variety $\WR_{L/F}X(1;D)$ is indecomposable.
\end{thm}

As explained above, the choice of settings for Theorems \ref{thm1} and \ref{thm2} is mainly motivated by possible applications to
unitary involutions.
Here are some more arguments motivating this choice.

The assumption that the field extension $L/F$ is separable
(which, in particular, insures that the Weil transfer of a projective variety is again a projective variety)
allows us to use the motivic Weil transfer functor constructed in \cite{MR1809664}.

We do not consider extensions $L/F$ of degree $>2$.
The case of a quadratic extension seems to be the natural case to start with.
Moreover, the answer in the case of $[L:F]>2$ should depend on many initial parameters
(like relations between the conjugate algebras of $D$).
At the same time, a separable extension of degree $>2$ can be non-galois, which makes the picture even more complicated.
One needs a really strong motivation, which we do not have now, to attack such a situation.

Since we stay with quadratic extensions, we do not consider canonical $p$-dimension (and $p$-incompressibility) for odd primes
$p$.
Indeed, since $\cdim_p$ is not changed under finite $p$-coprime field extensions by \cite[Proposition 1.5]{ASM-ed},
the canonical $p$-dimension of the Weil transfer $\WR(X)$ of any $L$-variety $X$ coincides -- if the prime $p$ is odd --
with the canonical $p$-dimension of the $L$-variety $\WR(X)_L$ which is isomorphic to the product of $X$ by its conjugate.
Therefore, the problem of computing the canonical $p$-dimension  (still interesting in some cases and trivial in some others, like in the
trivial norm algebra case, our main case of consideration here) has not much to do with the Weil transfer
anymore and is
a problem concerning canonical $p$-dimension of products.
We refer to \cite{MR2393078} where such a problem is addressed, partially solved, and applied.

Now let $p$ be any prime, $L/F$ an arbitrary finite separable field extension, and $A$ an arbitrary central simple $L$-algebra.
For any generalized Severi-Brauer variety $X$ of $A$, the canonical $p$-dimension of $\WR_{L/F}X$ can be
easily computed in terms of $\cdim_p\WR_{L/F}X(p^i;D)$, where $D$ is a central division $L$-algebra Brauer-equivalent to
the $p$-primary part of $A$ and where $i$ runs over the non-negative integers satisfying $p^i<\deg D$, see Lemma \ref{cdp}.
Therefore, Theorem \ref{thm1} determies $\cdim_p\WR_{L/F}X$ in the case of $p=2=[L:F]$ and of trivial $N_{L/F}A$.

Finally, the assumption on the norm of $D$ made in Theorems \ref{thm1} and \ref{thm2}, can be avoided in the
usual Severi-Brauer variety case.
It turns out that a generic field extension killing the norm algebra $N_{L/F}D$ does not affect the canonical $2$-dimension
of the variety $\WR_{L/F}X(1;D)$, see Proposition \ref{no norm}.
As a consequence, the following three conditions are equivalent:
\begin{enumerate}
\item
the variety $\WR_{L/F}X(1;D)$ is $2$-incompressible;
\item
the $2$-motive of the variety $\WR_{L/F}X(1;D)$ is indecomposable;
\item
the division algebra $D$ remains division over
the function field of the variety $X(1;N_{L/F}D)$.
\end{enumerate}

In Section \ref{Motivic Weil transfer} we discuss the motivic functor given by the Weil transfer.
In Section \ref{Some motivic decompositions} we write down some simple or known motivic
decompositions used afterwards.
In Section \ref{Trivial norm algebras}, both Theorems \ref{thm1} and \ref{thm2} are proved.
Some easy generalizations are considered in Section \ref{Generalizations}.

\medskip
\noindent
{\sl Acknowledgments:}
This work was done during my stay at the Euler Institute, St.Petersburg, in summer 2009;
many thanks to the Institute for the superb working conditions.
The results were reported at the Linear Algebraic Groups and Related Structures workshop held
in September 2009 at the
Banff International Research Station; my gratitude goes to the Station, the organizers of the workshop and to its
participants.

\section
{Motivic Weil transfer}
\label{Motivic Weil transfer}

Let $F$ be a field.
We fix a quadratic separable field extension $L/F$ and
we write $\WR_{L/F}X$ (or simply $\WR X$) for the Weil transfer of an $L$-variety $X$
(see \cite{MR1809664} for the definition and basic properties of $\WR$ as well as for further references on it).
We are working with the category $\CM(F,\Lambda)$ (constructed -- in contrast to \cite{EKM} -- out of smooth
{\em projective}, not just complete, $F$-varieties) of the Chow $F$-motives with coefficients in an associative
unital commutative ring $\Lambda$, \cite[\S64]{EKM}
(we will set $\Lambda=\F_2$, the field of $2$ elements, in the next section).
We recall that the Weil transfer extends to motives giving a (non-additive and not commuting with the Tate shift)
functor $\CM(L,\Lambda)\to\CM(F,\Lambda)$ of the category of $L$-motives into the category of $F$-motives, \cite{MR1809664}.
We write $\cores_{L/F}$ (or simply $\cores$) for the (additive and commuting with the Tate shift)
functor $\CM(L,\Lambda)\to\CM(F,\Lambda)$, studied in
\cite{outer}, associating to the motive of an $L$-variety $X$ the motive of the $F$-variety $\cores X$, which is
is the scheme $X$ considered as an $F$-variety via the composition $X\to\Spec L\to\Spec F$.

Finally, $\sigma$ is the non-trivial automorphism of $L/F$;
$\sigma X$, the {\em conjugate} of $X$,  is the base change of $X$ by $\sigma:\Spec L\to\Spec L$,
and $\sigma:\CM(L,\Lambda)\to\CM(L,\Lambda)$ is the induced motivic (conjugation) functor.

\begin{lemma}
\label{R(sum)}
For any two $L$-motives $M$ and $N$ one has
$$
\WR(M\oplus N)\simeq \WR(M)\oplus\cores(M\otimes \sigma N)\oplus\WR(N).
$$
Besides, $\WR(M\otimes N)\simeq \WR(M)\otimes\WR(N)$ and $\WR(\F_2(i))\simeq\F_2(2i)$ for any integer $i$.
\end{lemma}

\begin{proof}
The formulas for a tensor product and for a Tate motive are in \cite[Theorem 5.4]{MR1809664}.
We only have to prove the first formula.

Let us start by assuming that $M$ and $N$ are the motives of some $L$-varieties $X$ and $Y$.
We recall that the $F$-variety $\WR X$ is determined (up to an isomorphism) by the fact that there exists a $\sigma$-isomorphism
(that is, an isomorphism commuting with the action of $\sigma$) of the $L$-varieties $(\WR X)_L$ and $X\times \sigma X$.
At the same time, the $F$-variety $X$, which we denote as $\cores X$, is determined (up to an isomorphism) by existence
of a $\sigma$-isomorphism
of the $L$-varieties $(\cores X)_L$ and $X\coprod \sigma X$.

Since the $L$-varieties
$$
(X\Coprod Y)\times\sigma(X\Coprod Y)\;\;
\text{ and }\;\;
(X\times X)\Coprod(X\times\sigma Y)\Coprod\sigma(X\times\sigma Y)\Coprod(Y\times\sigma Y)
$$
are $\sigma$-isomorphic,
it follows that
$$
\WR(X\Coprod Y)\simeq(\WR X)\Coprod(\cores X\Coprod\sigma Y)\Coprod(\WR Y),
$$
whence the motivic formula.

In the general case, we have $M=(X,[\pi])$ and $(Y,[\tau])$ for some algebraic cycles $\pi$ and $\tau$
($[\pi]$ and $[\tau]$ are their classes modulo rational equivalence).
Using the same letter $\WR$ also for the Weil transfer of algebraic cycles, defined in \cite[\S3]{MR1809664},
as well as for the Weil transfer of their classes, defined in \cite[\S4]{MR1809664},
and processing similarly with the notation $\cores$ and $\sigma$,
we get the following formula to check:
$$
\WR([\pi]+[\tau])=\WR [\pi]+\cores([\pi]+\sigma[\tau])+\WR [\tau].
$$
This formula is easy to check because it holds already on the level of algebraic cycles, that is, with
$[\pi]$ and $[\tau]$ replaced by $\pi$ and $\tau$.
Since the group of algebraic cycles of any $F$-variety injects into the group of algebraic cycles of the same variety considered
over $L$, it suffices to check the latter formula over $L$, where it becomes the trivial relation
\begin{equation*}
(\pi+\tau)\cdot\sigma(\pi+\tau)=\pi\cdot\sigma\pi+\pi\cdot\sigma\tau+\tau\cdot\sigma\pi+\tau\cdot\sigma\tau.
\qedhere
\end{equation*}
\end{proof}

\section
{Some motivic decompositions}
\label{Some motivic decompositions}

Starting from this section, we are working with Chow motives with coefficients in the finite field $\F_2$, that is, we set
$\Lambda=\F_2$ in the notation of the previous section.
Therefore the Krull-Schmidt principle holds for the motives of the projective homogeneous varieties
(see \cite{MR2264459} or \cite{upper}).
This means that any summand of the motive of a projective homogeneous variety possesses a finite direct sum decomposition with
indecomposable summands and such a decomposition (called {\em complete} in this paper) is unique in the usual sense.

First we recall several known facts about the motives of Severi-Brauer varieties.

Let $F$ be a field, $D$ a central simple $F$-algebra, and let $X$ be the Severi-Brauer variety $X(1;D)$ of $D$.
If $D$ is a division algebra, then the motive $M(X)$ of $X$ is indecomposable.
The original proof of this fact is in \cite{MR1356536}, a simpler recent proof can be found in \cite{upper}.

Now, let us assume that $D$ is the algebra of ($2\times2$)-matrices over a central simple $F$-algebra $C$ and set
$Y=X(1;C)$.
Then, according to  \cite{MR1356536}, the  motive of $X$ decomposes in a sum of shifts of the motive of $Y$, namely,
$$
M(X)\simeq M(Y)\oplus M(Y)(\deg C).
$$
(This is, of course, a particular case of a general formula on the case where $D$ is the algebra of ($r\times r$)-matrices over $C$ for some
$r\geq2$.)

Finally, for arbitrary $D$, let $D'$ be one more central simple $F$-algebra (we will have $\deg D'=\deg D$ in the application),
$X':=X(1;D')$, and
assume that the class of $D'$ in the Brauer group $\Br F$ belongs to the subgroup generated by the class of $D$.
Then the projection $X\times X'\to X$ is a projective bundle and therefore
$$
M(X\times X')=M(X)\otimes M(X')\simeq\bigoplus_{i=0}^{\dim X'} M(X)(i)
$$
by the motivic projective bundle theorem, \cite{EKM}.

A known consequence of this decomposition is as follows.
Assume that $D$ and $D'$ are division algebras whose classes generate the same subgroup in $\Br F$
(for instance, $D'$ can be the {\em opposite} algebra of $D$).
Then $M(X\times X')$ is also expressed in terms of $M(X')$ and it follows by the Krull-Schmidt principle that
$M(X)\simeq M(X')$.
A different proof, working in a more general case of a generalized Severi-Brauer variety is given shortly below
(in the last paragraph before Lemma \ref{iso deco}).

Now let us describe the similar results concerning the {\em generalized} Severi-Brauer varieties.
Since we are only interested in the $2$-primary algebras in this paper and for the sake of simplicity, we assume that $D$ is
a central simple $F$-algebra of degree $2^n$ with some $n\geq0$.
Let $X=X(2^i;D)$ with some $i$ satisfying $0\leq i\leq n$.
If $D$ is division, then the variety $X$ is $2$-incompressible (though the motive of $X$ is usually {\em decomposable}
for $i\ne0,1,n$ by \cite{maksim}).
In motivic terms, the $2$-incompressibility of $X$ is expressed as follows: the indecomposable {\em upper} summand $M_X$ of $M(X)$
is {\em lower}.
The adjective {\em upper}, introduced in \cite{upper}, simply means that the $0$-codimensional Chow group of $M_X$ is non-zero.
By the Krull-Schmidt principle, the motive $M_X$ is unique up to an isomorphism.
The adjective {\em lower}, also introduced in \cite{upper},
means that the $d$-dimensional Chow group of $M_X$ is non-zero, where $d=\dim X$.
This notion is dual to the notion of upper: the dual of an upper summand is lower and vice versa.
Therefore, the $2$-incompressibility of $X$ also means that the summand $M_X$ is self-dual.

Now let $D$ and $D'$ be central simple $F$-algebras whose classes in $\Br F$ generate the same subgroup.
Let $M$ and $M'$ be the upper indecomposable motivic summands of the varieties  $X:=X(r;D)$ and $X':=X(r';D')$,
where  $r$ and $r'$ are integers satisfying $0\leq r\leq\deg D$, $0\leq r'\leq\deg D'$, and $\gcd(r,\ind D)=\gcd(r',\ind D')$.
Then $X(F(X'))\ne\emptyset\ne X'(F(X))$, and it follows by \cite{upper} that $M\simeq M'$.

\begin{lemma}
\label{iso deco}
Fix integers $i$ and $n$ satisfying $0\leq i\leq n-1$.
Let $D$ be a central division $F$-algebra of degree $2^n$ and let $U$ be the upper indecomposable motive of the variety $X(2^i;D)$.
Let $K/F$ be a field extension and $C$ a central division $K$-algebra such that $D_K$ is isomorphic to the algebra of
($2\times2$)-matrices over $C$.
For any integer $j$ with $0\leq j\leq n-1$, let $V_j$ be the upper indecomposable motive of the $K$-variety $X(2^j;C)$ and $V:=V_i$.
Then the complete motivic decomposition of  $U_K$ contains $V$ and $V(2^{i+n-1})$, while each of the remaining summands
is a shifts of $V_j$ with some $j<i$.
\end{lemma}

\begin{proof}
Let $X=X(2^i;D)$ and $Y=X(2^i;C)$.
Note that there exist rational maps in both directions between the $K$-varieties
$X_K$ and $Y$.
Therefore the upper indecomposable motivic summands of the varieties $Y$ and $X_K$ are isomorphic.
Since $V$ is upper indecomposable and $U_K$ is upper, $V$ is a summand of $U_K$.
Since $U$ and $V$ are self-dual
(they are self-dual because the varieties $X$ and $Y$ are $2$-incompressible),
we get by dualizing that
$V(-\dim Y)$ is a summand of $U_K(-\dim X)$, that is,
$V(\dim X-\dim Y)$ is a summand of $U_K$.
Since $\dim X=2^i(2^n-2^i)$ and, similarly, $\dim Y=2^i(2^{n-1}-2^i)$,
we have $\dim X-\dim Y=2^{i+n-1}$.

According to \cite[Corollary 10.19]{MR1758562} (more general results of \cite{MR2110630} or of \cite{MR2178658}
can be used here instead),
the motive of $X_K$ decomposes in a sum over the integers $r\in[0,\;2^i]$ of the motives of the products
$X(r;C)\times X(2^i\!-\!r;C)$ with some shifts.
According to \cite{upper}, the complete motivic decomposition of the product
$X(r;C)\times X(2^i\!-\!r;C)$ consists of shifts of the motives $V_j$ with $2^j|r$.
A shift of $V$ appears only two times: one time for $r=0$ and another time for $r=2^i$.
Since the remaining indecomposable summands of $M(X_K)$ are shifts of $V(j)$ with $j<i$, the remaining indecomposable
summands of $U_K$ are also shifts of $V(j)$ with $j<i$.
\end{proof}

\begin{lemma}
\label{pro deco}
For $i$, $n$, $D$, and $U$ as in Lemma \ref{iso deco},
let $U_j$ with $0\leq j\leq n$ be the upper indecomposable motivic summand of $X(2^j;D)$.
The complete decomposition of the tensor product $U\otimes U$ contains $U$ and $U(d)$, where $d=2^i(2^n-2^i)=\dim X(2^i;D)$,
while the remaining summands are
$U(j)$ with some $j\in[1,\;d-1]$ and shifts of $U_j$ with some $j<i$.
\end{lemma}

\begin{proof}
Let $X=X(2^i;D)$.
Since there exist rational maps in both directions between the varieties
$X\times X$ and $X$, $U$ is an upper indecomposable summand of $M(X\times X)$.
Since $U\otimes U$ is an upper summand of $M(X\times X)$, it follows that $U$ is a summand of $U\otimes U$.
Dualizing, we get that $U(d)$ is a summand of $U\otimes U$.

According to \cite{upper}, the complete motivic decomposition of $X\times X$ consists of shifts of $U_j$ with
$j\leq i$.
Since the summand $U$ is upper while $U(d)$ is lower, a summand $U(j)$ with some $j$ is present among the remaining summands
of $M(X\times X)$ only if $j\in[1,\;d-1]$.
Since $U\otimes U$ is a summand of $M(X\times X)$, the same statement holds for $U\otimes U$.
\end{proof}

\section
{Proofs of the main theorems}
\label{Trivial norm algebras}

We recall that
we are working with the Chow motives with coefficients in the finite field $\F_2$, that is,
$\Lambda=\F_2$ in the notation of Section \ref{Motivic Weil transfer}.

We start proving Theorem \ref{thm2}:

\begin{proof}[Proof of Theorem \ref{thm2}]
We set $X:=X(1;D)$.

We induct on $n$.
For $n=0$ we have $X=\Spec L$, $\WR X=\Spec F$, and the statement is trivial.
Below we are assuming that $n>0$ and that the statement holds for all fields and all central division algebras
of degree $2^{n-1}$.

Let $M$ be an upper motivic summand of $\WR X$
(the adjective {\em upper} is defined in Section \ref{Some motivic decompositions}).
It suffices to show that $M$ is the whole motive of $\WR X$.

Our proof is illustrated by Figure 1.
An explanation of the illustration is given right after the end of the proof.

We recall that the $L$-variety $(\WR X)_L$ is isomorphic to $X\times \sigma X$,
where $\sigma X$ is the conjugate variety.
Note that $\sigma X=X(1;\sigma D)$, where $\sigma D$ is the conjugate algebra.
Since $(N_{L/F}D)_L\simeq D\otimes_L\sigma D$ and the $F$-algebra $N_{L/F}D$ is trivial,
the $L$-algebra $\sigma D$ is opposite to $D$.
Therefore, 
as explained in Section \ref{Some motivic decompositions},
the complete motivic decomposition of the $L$-variety $(\WR X)_L\simeq X\times \sigma X$ looks as follows
(note that $\dim X= 2^n-1$):
$$
M(\WR X)_L\simeq
\bigoplus_{i=0}^{2^n-1}M(X)(i).
$$
Since the summand $M$ is upper, the complete decomposition of $M_L$ contains a summand isomorphic to $M(X)$.
(The Krull-Schmidt principle is used at this point and will be used several times later on in the proof.)

Let $E$ be the function field of the variety $\WR X(2^{n-1};D)$.
We write $EL$ for the field $E\otimes_F L$.
The $EL$-algebra $D\otimes_FE=D\otimes_L(EL)=D_{EL}$ is isomorphic to the algebra of ($2\times2$)-matrices over a central division
$EL$-algebra $C$ of degree $2^{n-1}$.
By Section \ref{Some motivic decompositions},
the motive of the $EL$-variety $X_{EL}$ decomposes as $N\oplus N(2^{n-1})$, where $N$
is the motive of $X(1;C)$.
Therefore the complete motivic decomposition of the $EL$-variety $(\WR X)_{EL}$ consists of the summands
$N(i)$, where the integer $i$ runs over the interval $[0,\;2^n+2^{n-1}-1]$.
More precisely, there are two copies of $N(i)$ for $i\in[2^{n-1},\;2^n-1]\subset[0,\;2^n+2^{n-1}-1]$
and one copy for $i\in[0,\;2^n+2^{n-1}-1]\setminus[2^{n-1},\;2^n-1]$.

Since $M_L$ contains $M(X)$, $M_{EL}$ contains $N$.

We are working with the fields of the diagram
$$
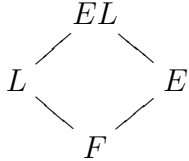
\xymatrix
@=10pt
{
&EL\ar@{-}[rd]\ar@{-}[ld]&\\
L\ar@{-}[rd]&&E\ar@{-}[ld]\\
&F&
}
$$
Recall  that the motive of the $EL$-variety $X_{EL}$ decomposes as $N\oplus N(2^{n-1})$.
Therefore, by Lemma \ref{R(sum)}, the motive of the $E$-variety $(\WR X)_E= \WR_{EL/E}(X_{EL})$
decomposes as
$$
\WR N\oplus\cores (N\otimes N')(2^{n-1})\oplus\WR N(2^n),
$$
where $\WR=\WR_{EL/E}$, $\cores=\cores_{EL/E}$, and where $N'$ is the conjugate of $N$.

The motive $\WR N$ is indecomposable  by the induction hypothesis.
Since
$$
N\otimes N'\simeq\bigoplus_{i=0}^{2^{n-1}-1}N(i)
$$
(by Section \ref{Some motivic decompositions}) and
the functor $\cores$ preserves indecomposability by \cite{outer},
we get the complete motivic decomposition of
$(\WR X)_E$ replacing the summand $\cores (N\otimes N')(2^{n-1})$ by the sum
$\bigoplus_{i=2^{n-1}}^{2^n-1}\cores N(i)$.

Since $M_{EL}$ contains $N$, $M_E$ contains $\WR N$ and it follows that
$M_{EL}$ contains $N(i)$ for $i=0,1,\dots,2^{n-1}-1$.
Looking at the
complete motivic decomposition of $(\WR X)_L$, we see that
$M_L$ contains $M(X)(i)$ for such $i$ and therefore $M_{EL}$ also contains $N(2^{n-1}+i)$.

It follows that $M_E$ contains $\cores N(i)$ for $i=2^{n-1},\dots,2^n-1$.
Since $(\cores N)_{EL}\simeq N\oplus N'\simeq N\oplus N$
(note that $N'\simeq N$ by Section \ref{Some motivic decompositions}),
$M_{EL}$ contains both copies of $N(i)$ for such $i$.
We conclude that $M_L$ contains $M(X)(i)$ also for $i=2^{n-1},\dots,2^n-1$.
Consequently, $M_L=M(\WR X)_L$.
It follows that $M=M(\WR X)$ and Theorem \ref{thm2} is proved.
\end{proof}

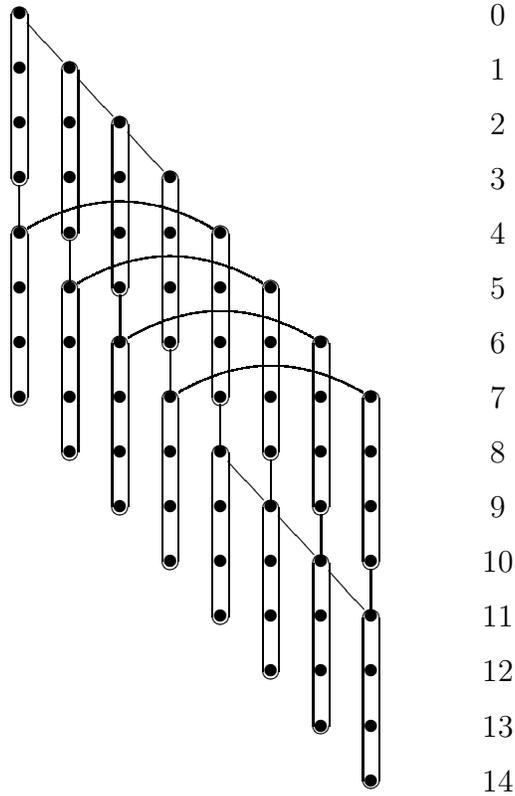
\begin{figure*}
$$
\xymatrix@M=0pt@=13pt{
\bullet&&&&&&&&&&0\\
\bullet&\bullet&&&&&&&&&1\\
\bullet&\bullet&\bullet&&&&&&&&2\\
\bullet&\bullet&\bullet&\bullet&&&&&&&3\\
\bullet&\bullet&\bullet&\bullet&\bullet&&&&&&4\\
\bullet&\bullet&\bullet&\bullet&\bullet&\bullet&&&&&5\\
\bullet&\bullet&\bullet&\bullet&\bullet&\bullet&\bullet&&&&6\\
\bullet&\bullet&\bullet&\bullet&\bullet&\bullet&\bullet&\bullet&&&7\\
&\bullet&\bullet&\bullet&\bullet&\bullet&\bullet&\bullet&&&8\\
&&\bullet&\bullet&\bullet&\bullet&\bullet&\bullet&&&9\\
&&&\bullet&\bullet&\bullet&\bullet&\bullet&&&10\\
&&&&\bullet&\bullet&\bullet&\bullet&&&11\\
&&&&&\bullet&\bullet&\bullet&&&12\\
&&&&&&\bullet&\bullet&&&13\\
&&&&&&&\bullet&&&14
\save "1,1"."4,1"="U1"*\frm<44pt>{-}
\restore
\save "2,2"."5,2"="U2"*\frm<44pt>{-}
\restore
\save "3,3"."6,3"="U3"*\frm<44pt>{-}
\restore
\save "4,4"."7,4"="U4"*\frm<44pt>{-}
\restore
\save "5,5"."8,5"="U5"*\frm<44pt>{-}
\restore
\save "6,6"."9,6"="U6"*\frm<44pt>{-}
\restore
\save "7,7"."10,7"="U7"*\frm<44pt>{-}
\restore
\save "8,8"."11,8"="U8"*\frm<44pt>{-}
\restore
\save "5,1"."8,1"="L1"*\frm<44pt>{-}
\restore
\save "6,2"."9,2"="L2"*\frm<44pt>{-}
\restore
\save "7,3"."10,3"="L3"*\frm<44pt>{-}
\restore
\save "8,4"."11,4"="L4"*\frm<44pt>{-}
\restore
\save "9,5"."12,5"="L5"*\frm<44pt>{-}
\restore
\save "10,6"."13,6"="L6"*\frm<44pt>{-}
\restore
\save "11,7"."14,7"="L7"*\frm<44pt>{-}
\restore
\save "12,8"."15,8"="L8"*\frm<44pt>{-}
\restore
\ar@{-}"L1";"U1"
\ar@{-}"L2";"U2"
\ar@{-}"L3";"U3"
\ar@{-}"L4";"U4"
\ar@{-}"L5";"U5"
\ar@{-}"L6";"U6"
\ar@{-}"L7";"U7"
\ar@{-}"L8";"U8"
\ar@{-}"U1";"U2"
\ar@{-}"U2";"U3"
\ar@{-}"U3";"U4"
\ar@{-}"L5";"L6"
\ar@{-}"L6";"L7"
\ar@{-}"L7";"L8"
\ar@/^1pc/@{-}"L1";"U5"
\ar@/^1pc/@{-}"L2";"U6"
\ar@/^1pc/@{-}"L3";"U7"
\ar@/^1pc/@{-}"L4";"U8"
}
$$
\caption{Proof of Theorem \ref{thm2}}
\end{figure*}

Figure 1 illustrates the proof, just finished, in the case of $n=3$.
The ovals represent the summands of a complete motivic decomposition of the variety
$(\WR X)_{EL}$.
Note that two {\em different} complete motivic decompositions of $(\WR X)_{EL}$ have been used in the proof:
one is a refinement of a complete motivic decomposition of $(\WR X)_L$, the other of $(\WR X)_E$.
But because of the Krull-Schmidt principle,
the sets of summands of these two decompositions can be identified in such a way that the identified summands are isomorphic.
All the summands are shifts of $N$: there is one copy of $[N(i)]$ for $i=0,1,2,3,8,9,10,11$ and two copies
of $[N(i)]$ for $i=4,5,6,7$.
Over an algebraic closure of $EL$, each summand $N(i)$ becomes the sum of the Tate motives $\F_2(j)$ with
$j=i,i+1,i+2,i+3$;
these are indicated by the bullets inside of the ovals (the numbers on the right indicating the shifts
of the corresponding Tate motives).
The connection lines between the ovals are there to show that the connected ovals are inside of the same indecomposable summand
over a smaller field: the vertical lines are the connections coming from the field $L$, while the others are from the field $E$
with the straight lines coming from $\WR N$ due to the induction hypothesis and the curved lines coming from
$\cores N$.
Note that the identification mentioned above (which arises from the Krull-Schmidt principle) does not preserves connections.
So, it is a lucky coincidence that all pairs of isomorphic summands in each of two decompositions are $E$-connected
(otherwise we would not be able to decide
which of two isomorphic summands to use when drawing the $L$-connections;
because of the coincidence the choice does not matter).
Since all the ovals turn out to be connected by (a chain of) connection lines, and $F$ is a subfield of both $L$ and $E$,
the motive of the $F$-variety $\WR X$ is indecomposable.

We come to the proof of Theorem \ref{thm1} now:

\begin{proof}[Proof of Theorem \ref{thm1}]
We fix the integer $i\geq0$ and
we induct on $n\geq i$.
For $n=i$ the statement is trivial.
Below we are assuming that $n>i$ and that the statement (with the fixed $i$) holds for
all fields and all central division algebras
of degree $2^{n-1}$.

Let $U$ be the upper (see Section \ref{Some motivic decompositions})
indecomposable motivic summand of the variety $X(2^i;D)$.
Since this variety is $2$-incompressible, the summand $U$ is lower
(see Section \ref{Some motivic decompositions}).
Note that by Lemma \ref{R(sum)}, $\WR U$ is a motivic summand of the variety $\WR X(2^i;D)$.
Let $M$ be a summand of $\WR U$ which is upper as a motivic summand of $\WR X(2^i;D)$.
It suffices to show that $M$ is lower.

Our proof is illustrated by Figure 2 right after the end of the proof.
An explanation of the illustration is given in the end of the proof.

The $L$-variety $\big(\WR X(2^i;D)\big)_L$ is isomorphic to $X(2^i;D)\times X(2^i;\sigma D)$,
and the conjugate algebra $\sigma D$ is opposite to $D$.
Therefore the upper indecomposable motivic summand of this variety is isomorphic to $U$
and the lower indecomposable motivic summand of this variety is isomorphic to
$U\big(\dim X(2^i;D)\big)=U\big(2^i(2^n-2^i)\big)$.
Since the summand $M$ is upper, $M_L$ contains $U$.
Our aim is to show that $M_L$ contains $U\big(2^i(2^n-2^i)\big)$.

As in the proof of Theorem \ref{thm2}, we write
$E$ for the function field of the variety $\WR X(2^{n-1},D)$ and we are working with the diagram
of fields
$$
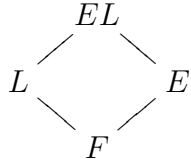
\xymatrix
@=10pt
{
&EL\ar@{-}[rd]\ar@{-}[ld]&\\
L\ar@{-}[rd]&&E\ar@{-}[ld]\\
&F&
}
$$
The $EL$-algebra $D_{EL}$ is isomorphic to the algebra of
($2\times2$)-matrices over a central division
$EL$-algebra $C$ of degree $2^{n-1}$.
By Lemma \ref{iso deco},
the motive $U_{EL}$ decomposes as
$V\oplus V(2^{i+n-1})\oplus\dots$, where $V$ is the upper motive of the variety
$X(2^i;C)$ and $\dots$ stands for a sum of the upper motives of varieties $X(2^j;C)$ with
$j<i$.
By Lemma \ref{pro deco}, if $i<n-1$, the
tensor product $V\otimes V$ decomposes as $V\oplus V(d)\oplus ?\oplus\dots$,
where $d=\big(2^i(2^{n-1}-2^i)\big)$,
$?$ is a sum of $V(j)$ with $j\in[1,\;d-1]$, and where $\dots$ stands
for a sum of the upper motives of varieties $X(2^j;C)$ with $j<i$.
Therefore the complete motivic decomposition of $(\WR U)_{EL}$
looks as
\begin{multline*}
V\oplus V(d)
\oplus V(2^{i+n-1}) \oplus V(2^{i+n-1}+d)
\oplus V(2^{i+n-1}) \oplus V(2^{i+n-1}+d)
\\
\oplus V(2^{i+n})\oplus V(2^{i+n}+d)
\oplus ?\oplus\dots.
\end{multline*}
Here $?$ is a sum of $V(j)$ with $j$ in the (disjoint) union of the intervals
$$
[1,\;d-1]\cup[2^{i+n-1}+1,\;2^{i+n-1}+d-1]\cup[2^{i+n}+1,\;2^{i+n}+d-1].
$$
(The three intervals are pairwise disjoint because $d<2^{i+n-1}$.)
And, as before, $\dots$ stands for a sum of the upper motives of varieties $X(2^j;C)$ with $j<i$.

In the case of $i=n-1$, we have $d=0$, $V$ is the Tate motive $\F_2$, and $V\otimes V=V$.
Therefore each of the pairs of summands $V(r)$, $V(r+d)$ for $r=0,2^{i+n-1}\text{ (two times)}, 2^{i+n}$ in the complete motivic
decomposition of $(\WR U)_{EL}$ indicated above is replaced by a single $V(r)$ (and the sum $?$ is empty).

By now we know that $M_{EL}$ contains $V$ and (at least one copy of) $V(2^{i+n-1})$.

Recall  that $U_{EL}$ decomposes as $V\oplus V(2^{i+n-1})\oplus\dots$.
Therefore, by Lemma \ref{R(sum)}, $(\WR U)_E= \WR_{EL/E}U_{EL}$
decomposes as
$$
\WR V\oplus\cores (V\otimes V)(2^{i+n-1})\oplus\WR V(2^{i+n})\oplus\dots.
$$

Any upper summand of $\WR V$ is lower by the induction hypothesis.
It follows that $M_{EL}$ contains $V(d)$.
Therefore $M_{EL}$ contains (at least one copy of) $V(2^{i+n-1}+d)$.
It follows that $M_E$ contains $\cores V(2^{i+n-1}+d)$.
Therefore $M_{EL}$ contains both copies of $V(2^{i+n-1}+d)$ and it finally follows that
$M_{EL}$ contains $V(2^{i+n}+d)$ and is lower.
\end{proof}

\begin{figure*}
$$
\xymatrix@M=0pt@=13pt{
\;\;\cdot\;\;&&&\\
\dots&&&\\
\cdot&\;\;\cdot\;\;&&\\
\;\;\cdot\;\;&\dots&\;\;\cdot\;\;&\\
\dots&\cdot&\dots&\\
\cdot&\;\;\cdot\;\;&\cdot&\;\;\cdot\;\;\\
&\dots&\;\;\cdot\;\;&\dots\\
&\cdot&\dots&\cdot\\
&&\cdot&\;\;\cdot\;\;\\
&&&\dots\\
&&&\cdot
\save "1,1"."3,1"="U1"*\frm<44pt>{-}
\restore
\save "4,1"."6,1"="L1"*\frm<44pt>{-}
\restore
\save "3,2"."5,2"="U2"*\frm<44pt>{-}
\restore
\save "6,2"."8,2"="L2"*\frm<44pt>{-}
\restore
\save "4,3"."6,3"="U3"*\frm<44pt>{-}
\restore
\save "7,3"."9,3"="L3"*\frm<44pt>{-}
\restore
\save "6,4"."8,4"="U4"*\frm<44pt>{-}
\restore
\save "9,4"."11,4"="L4"*\frm<44pt>{-}
\restore
\ar@{-}"L1";"U1"
\ar@{-}"L2";"U2"
\ar@{-}"L3";"U3"
\ar@{-}"L4";"U4"
\ar@{-}"U1";"U2"
\ar@{-}"L3";"L4"
\ar@/^0.7pc/@{-}"L1";"U3"
\ar@/^0.7pc/@{-}"L2";"U4"
}
$$
\caption{Proof of Theorem \ref{thm1}}
\end{figure*}
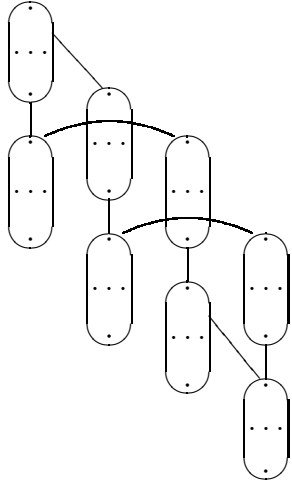

Figure 2 illustrates the proof, just finished (in the case of $d>0$).
The ovals represent some summands
in a complete motivic decomposition of the variety $(\WR X)_{EL}$:
one copy of $[V]$, $[V(d)]$, $[V(2^{i+n})]$, $[V(2^{i+n}+d)]$ and two copies of
$[V(2^{i+n-1})]$ and of $[V(2^{i+n-1}+d)]$.
None of the remaining summand of the decomposition is isomorphic to a represented one.
Note that two {\em different} complete motivic decompositions of the variety $\WR X(2^i;D)$ over the field $EL$
have been used in the proof:
one is a refinement of a complete motivic decomposition over $L$, the other over $E$.
But the sets of summands of these decompositions can be identified because of the Krull-Schmidt principle.

In contrast with Figure 1, where the summands are ``thin'', those of Figure 2 are, in general, ``thick'', that is,
contain (over an algebraic closure of $EL$) several copies of the Tate motives with a same shift number.
Since the picture illustrates the general case, we cannot (and do not) indicate the single Tate motives inside of the ovals.
We only take care of representing a summand with a bigger shift number by a lower position oval.
The connection lines between the ovals are there to show that the connected ovals are inside of the same indecomposable summand
over a smaller field: the vertical lines are the connections coming from the field $L$, while the others are from the field $E$
with the straight lines coming from $\WR V$ due to the induction hypothesis and the curved lines coming from the
$\cores V$.
Note that the mentioned above identification (which arises from the Krull-Schmidt principle) does not preserves connections.
So again, it is a lucky coincidence that
all pairs of isomorphic summands in each of two decompositions are $E$-connected (otherwise we would not be able to decide
which of two isomorphic summands to use when drawing the $L$-connections).

Since the upper oval turns out to be connected by (a chain of) connection lines with the lower one,
the variety $\WR X$ is $2$-incompressible.

\section
{Generalizations}
\label{Generalizations}

Let $p$ be any prime, $L/F$ an arbitrary finite separable field extension, and $A$ an arbitrary central simple $L$-algebra.
For any integer $r$ with $0\leq r\leq\deg A$, the canonical $p$-dimension of $\WR_{L/F}X(r;A)$ can be
easily computed in terms of $\cdim_p\WR_{L/F}X(p^i;D)$, where $D$ the central division $L$-algebra Brauer-equivalent to the
$p$-primary part of $A$ and where $i$ runs over the non-negative integers satisfying $p^i<\deg D$:

\begin{lemma}
\label{cdp}
In the above settings, we have
$$
\cdim_p \WR_{L/F}X(r;A)=\cdim_p\WR_{L/F}X(p^i;D),
$$
where $i=\min\{v_p(r),\;v_p(\deg D)\}$ with $v_p(\cdot)$ standing for the $p$-adic order.
\end{lemma}

\begin{proof}
The variety $\WR_{L/F}X(p^i;D)$ has a point over the function field of
$\WR_{L/F}X(r;A)$:
\begin{multline*}
\Big(\WR_{L/F}X(p^i;D)\Big)\Big(F\big(\WR_{L/F}X(r;A)\big)\Big)=\\
X(p^i;D)\Big(L\big(X(r;A)\big)\otimes_LL\big(X(r;\sigma A)\big)\Big)\supset
X(p^i;D)\Big(L\big(X(r;A)\big)\Big)\ne\emptyset
\end{multline*}
because $\ind D_{L(X(r;A))}$ divides $p^i$.
Similarly, the variety $\WR_{L/F}X(r;A)$ has a point over a finite $p$-coprime extension of the function field of
$\WR_{L/F}X(p^i;D)$.
\end{proof}

We turn back to the case of $p=2=[L:F]$ in order to remove the norm condition of Theorems \ref{thm1} and \ref{thm2}
in the case of a usual Severi-Brauer variety.
Note that the function field $F(Y)$ in the following statement is a (generic) splitting field of the norm algebra of $D$:

\begin{prop}
\label{no norm}
Let $L/F$ be a quadratic separable field extension,
$D$ a $2$-primary central division $L$-algebra, $X=\WR_{L/F}X(1;D)$, and $Y=X(1;N_{L/F}D)$.
Then $\cdim_2 X=\cdim_2 X_{F(Y)}$.
\end{prop}

\begin{proof}
Note that $X_{F(Y)}\simeq\WR_{L(Y)/F(Y)}X(1;D_{L(Y)})$ and $Y_L\simeq X(1;D\otimes_LD')$,
where $D'$ is the conjugate algebra $\sigma D$.
%
According to the index reduction formula of \cite{MR1061778},
$\ind D_{L(Y)}$ is equal to the minimum of $\ind\big(D^{\otimes (i+1)}\otimes (D')^{\otimes i}\big)$
where $i$ runs over the integers.
Let $i$ be an integer which gives the minimum and let
$\tilde{D}$ be a central division $L$-algebra Brauer-equivalent to the product
$D^{\otimes (i+1)}\otimes (D')^{\otimes i}$.
We set $\tilde{X}=\WR X(1;\tilde{D})$.
Since the algebra $\tilde{D}':=\sigma\tilde{D}$ is Brauer-equivalent to the product
$(D')^{\otimes (i+1)}\otimes D^{\otimes i}$ and
the exponent of $D$ (coinciding with the exponent of $D'$) is a power of $2$,
the classes of $\tilde{D}$ and $\tilde{D'}$ in $\Br(L)$ generate the same subgroup
as the classes of $D$ and $D'$.
It follows that
$$
\tilde{X}\big(F(X)\big)\ne\emptyset\ne X\big(F(\tilde{X})\big).
$$
Therefore, for any field extension $E/F$, we have $\cdim_2\tilde{X}_E=\cdim_2 X_E$.
On the other hand, for $\tilde{Y}=X(1;N \tilde{D})$, the algebra
$\tilde{D}_{L(\tilde{Y})}$ is division.
Consequently, by Theorem \ref{thm1}, the variety $\tilde{X}_{F(\tilde{Y})}$
is $2$-incompressible.
Therefore the variety $\tilde{X}$ is also $2$-incompressible,
$\cdim_2 \tilde{X}=\dim \tilde{X}=\cdim_2 \tilde{X}_{F(\tilde{Y})}$,
and we obtain that $\cdim_2 X=\cdim_2 X_{F(\tilde{Y})}$.
Since the norm algebra $ND$ becomes trivial over $F(\tilde{Y})$,
the field extension $F(\tilde{Y})(Y)/F(\tilde{Y})$ is purely transcendental, and
the required statement follows.
\end{proof}


\def\cprime{$'$}

\end{document}